\documentclass{amsart}
\RequirePackage[utf8]{inputenc}

\usepackage{amsmath,amsthm,amsfonts,amssymb,amscd,amsbsy,dsfont,hyperref,graphicx,tabularx}
\usepackage[all]{xy}
\usepackage{pgf}

\newcommand{\dd}{\mathrm{d}}
\newcommand{\Met}{\mathrm{Met}}
\newcommand{\scal}{\mathrm{scal}}

\newcommand{\R}{\mathds{R}}
\newcommand{\C}{\mathds{C}}
\newcommand{\Hr}{\mathds{H}}

\newcommand{\N}{\mathds{N}}

\newcommand{\Sp}{\mathrm{Sp}}

\newcommand{\U}{\mathrm{U}}

\newcommand{\Spin}{\mathrm{Spin}}

\newcommand{\vol}{\mathrm{vol}}
\newcommand{\Vol}{\mathrm{Vol}}

\newcommand{\g}{\mathtt g}
\newcommand{\h}{\mathtt h}
\renewcommand{\k}{\mathtt k}
\newcommand{\gt}{\mathtt g_t}
\newcommand{\hs}{\mathtt h_t}
\newcommand{\kt}{\mathtt k_t}

\hypersetup{
    pdftoolbar=true,
    pdfmenubar=true,
    pdffitwindow=false,
    pdfstartview={FitH},
    pdftitle={Bifurcation and local rigidity of homogeneous solutions to the Yamabe problem on spheres},
    pdfauthor={R. G. Bettiol and P. Piccione},
    pdfsubject={Primary: 53C30, 58J55; Secondary: 53C20, 53C24, 58E07, 58E50},
    pdfkeywords={},
    pdfnewwindow=true,
    colorlinks=true,
    linkcolor=blue,
    citecolor=blue,
    urlcolor=blue,
}

\theoremstyle{plain}\newtheorem{theorem}{Theorem}[]
\theoremstyle{plain}\newtheorem{lemma}[theorem]{Lemma}
\theoremstyle{plain}\newtheorem{proposition}[theorem]{Proposition}
\theoremstyle{plain}\newtheorem{corollary}[theorem]{Corollary}
\theoremstyle{plain}\newtheorem*{question}{\sc Question}

\theoremstyle{definition}\newtheorem{definition}[theorem]{Definition}

\theoremstyle{remark}\newtheorem{remark}[theorem]{Remark}
\theoremstyle{remark}

\allowdisplaybreaks
\numberwithin{equation}{section}
\numberwithin{theorem}{section}

\title[Yamabe problem on homogeneous spheres]{Bifurcation and local rigidity of homogeneous solutions to the Yamabe problem on spheres}
\author{R. G. Bettiol \and P. Piccione}

\address{\begin{tabular}{lll}
Department of Mathematics & & Departamento de Matem\'atica \\
University of Notre Dame & &Universidade de S\~ao Paulo \\
255 Hayes Healy Center & & Rua do Mat\~ao, 1010 \\
46556-4618 Notre Dame, IN, USA & & 05508-090 S\~ao Paulo, SP, Brazil\\
\emph{E-mail address}: {\tt rbettiol@nd.edu} & & \emph{E-mail address}: {\tt piccione@ime.usp.br}
\end{tabular}
}

\date{June 2012}
\thanks{The first named author is supported by the NSF grant DMS-0941615. The second named author is partially supported by Fapesp, S\~ao Paulo, Brazil, and by CNPq, Brazil.}

\begin{document}
\begin{abstract}
We study existence and non-existence of constant scalar curvature metrics conformal and arbitrarily close to homogeneous metrics on spheres, using variational techniques. This describes all critical points of the Hilbert-Einstein functional on such conformal classes, near homogeneous metrics. Both bifurcation and local rigidity type phenomena are obtained for $1$-parameter families of $\U(n+1)$, $\Sp(n+1)$ and $\Spin(9)$-homogeneous metrics. 
\end{abstract}

\maketitle

\vspace{-.5cm}
\section{Introduction}
Given a closed Riemannian manifold $(M,g)$, with $\dim M=m\geq3$, the Yamabe problem concerns the existence of constant scalar curvature metrics conformal to $g$. Solutions to this problem can be characterized variationally as critical points of the Hilbert-Einstein functional restricted to the conformal class $[g]$. The proof that such a solution always exist is a consequence of the successive works of Yamabe~\cite{Yam60}, Trudinger~\cite{Tru68}, Aubin~\cite{Aub76} and Schoen~\cite{Schoen84}, and provides minimizers of the Hilbert-Einstein functional in each conformal class. For instance, Einstein metrics are minima of the functional in their conformal class. In fact, except for round metrics on spheres, they are the \emph{unique} metrics in their conformal class having constant scalar curvature, see Obata~\cite{Oba72}. Anderson~\cite{And05} proved that, generically, the Hilbert-Einstein functional restricted to a conformal class has a unique critical point. Nevertheless, there are many other families of constant scalar curvature metrics that are critical points of this functional, but not necessarily minimizers. This raises the following interesting problem:

\begin{question}
Given a closed Riemannian manifold $(M,g)$, classify all critical points of the Hilbert-Einstein functional restricted to the conformal class $[g]$.
\end{question}

In this paper, we partially answer the above question when $(M,g)$ is the sphere $S^m$ equipped with a homogeneous metric (recall that any such metric is a solution to the Yamabe problem). Before giving details about our results, let us recall a few previous works on the above question. Schoen~\cite{Schoen91} proved existence of an increasing number of solutions, with large energy and Morse index, in the conformal class of the product $S^1(r)\times S^{m-1}$ of round spheres, as the radius $r$ of the circle tends to infinity. A remarkable result regarding non-uniqueness of solutions is due to Pollack~\cite{pollack}, that proved existence of arbitrarily $C^0$-small perturbations of any given metric, with arbitrarily large number of solutions in its conformal class. Regarding perturbations of the round metric on the sphere, Ambrosetti and Malchiodi~\cite{am} obtained the existence of at least $2$ solutions for certain deformations. Later, Berti and Malchiodi~\cite{bm} obtained the existence of a sequence of unbounded solutions (in the $L^\infty$ sense) among other multiplicity results, also for perturbations of the round metric. Recently, Lima, Piccione and Zedda~\cite{lpz} studied the family of metrics on a product $M_1\times M_2$ of compact Riemannian manifolds obtained by rescaling one of the factors, similarly to the case of $S^1(r)\times S^{m-1}$ mentioned above, obtaining \emph{bifurcation} of solutions.

In a natural continuation of the latter, we study bifurcation of solutions from families of homogeneous metrics on spheres. More precisely, our main results (Theorems~\ref{thm:s2n+1}, \ref{thm:s4n+3} and \ref{thm:s15}) distinguish homogeneous metrics that are a \emph{bifurcation point} (i.e., an accumulation point of other solutions conformal to homogeneous solutions) from homogeneous metrics that are \emph{locally rigid} (i.e., not a bifurcation point). This local rigidity means uniqueness in some neighborhood inside its conformal class. For a precise definition of these notions, see Definition~\ref{def:biflocrig}. Thus, these results completely describe the arrangement of the critical points of the Hilbert-Einstein functional on conformal classes of homogeneous metrics on $S^m$ in a vicinity (in the $C^{k,\alpha}$ topology) of the set of homogeneous metrics. We note that this was already known only for the round metric, which belongs to all families of homogeneous metrics and it is trivially a bifurcation value for all of them.

Homogeneous metrics on spheres were classified by Ziller~\cite{ziller} in 1982. All such metrics arise as a deformation of the round metric by scaling it in the direction of the fibers of a Hopf fibration (these manifolds are also often referred to as \emph{Berger spheres}). More precisely, the total space of each Hopf fibration
\begin{equation*}
S^1\to S^{2n+1}\to\C P^n, \quad S^3\to S^{4n+3}\to\Hr P^n, \quad S^7\to S^{15}\to S^8\left(\tfrac12\right)
\end{equation*}
admits a family of homogeneous metrics, depending on one parameter $t>0$, obtained by scaling the round metric by a factor $t^2$ in the subbundle tangent to the Hopf fibers. This gives rise to:
\begin{itemize}
\item[$\gt$,] a $1$-parameter family of $\U(n+1)$-homogeneous metrics on $S^{2n+1}$;
\item[$\hs$,] a $1$-parameter family of $\Sp(n+1)$-homogeneous metrics on $S^{4n+3}$;
\item[$\kt$,] a $1$-parameter family of $\Spin(9)$-homogeneous metrics on $S^{15}$.
\end{itemize}
In addition, $S^{4n+3}$ carries a $3$-parameter family of $\Sp(n+1)$-homogeneous metrics given by rescaling any left-invariant metric on the fibers $S^3$ in the same fashion. Any homogeneous metric on $S^m$ is isometric to one of the above. In this paper, we deal with metrics in each of the three $1$-parameter families $\gt,\hs$ and $\kt$, since bifurcation and local rigidity phenomena are naturally cast for $1$-parameter families.

We can now describe our results in more detail. We prove (Theorem~\ref{thm:s2n+1}) that every metric $\gt$, except the round metric $\g_1$, is a locally rigid solution to the Yamabe problem. We also prove (Theorems~\ref{thm:s4n+3} and \ref{thm:s15}) that, for each of the two families $\hs$ and $\kt$, there exist corresponding decreasing sequences $\{t^{\h}_q\}$ and $\{t^{\k}_q\}$ in $]0,1]$ that start at $t_0^{\h}=t_0^{\k}=1$ and converge to $0$ as $q\to+\infty$, such that
\begin{itemize}
\item[(i)] $\hs$ (respectively $\kt$) is a bifurcation point if $t=t_q^{\h}$ (respectively $t=t_q^{\k}$) for some $q$;
\item[(ii)] $\hs$ (respectively $\kt$) is locally rigid if $t\not\in\{t_q^{\h}\}$ (respectively $t\not\in\{t_q^{\k}\}$).
\end{itemize}
This means that $\{\gt,t\neq1\}$ are locally the only critical points of the Hilbert-Einstein functional on their conformal classes (hence locally the only solutions to the Yamabe problem); while for $\{\hs,t>0\}$ and $\{\kt,t>0\}$ there are infinitely many bifurcating branches of critical points issuing from the homogeneous branch at $\h_{t^{\h}_q}$ and $\k_{t^{\k}_q}$ (hence infinitely many other solutions to the Yamabe problem that accumulate on homogeneous solutions). Moreover, there is a \emph{break of symmetry} at all these bifurcation values, i.e., the bifurcating branches consist of non-homogeneous metrics. Such different behaviors of $\hs$ and $\kt$ compared to $\gt$ are in part explained by the fact that the scalar curvature function $\scal(\gt)$ is bounded from above, while $\scal(\hs)$ and $\scal(\kt)$ blow up as $t\to 0$.

The round metric of $S^m$ belongs to all three families $\gt,\hs$ and $\kt$, corresponding to $t=1$. In the conformal class of this metric, it was already well-known that the critical points of the Hilbert-Einstein functional form a manifold of dimension $m+1$. This implies that the round metric is a \emph{trivial} bifurcation point in all three families. This follows from the fact that the round sphere is the unique compact Riemannian manifold to admit non-isometric conformal transformations (see Obata~\cite{Oba72}) and given any such transformation $f\colon S^m\to S^m$, the pull-back of the round metric by $f$ has constant scalar curvature (see Schoen~\cite[Sec 2]{Schoen87}). In this way, to each non-isometric conformal transformation of $S^m$ corresponds a solution to the Yamabe problem that is conformal to the round metric.

Our results can also be understood from the viewpoint of dynamical systems, where \emph{bifurcation} means a topological or qualitative change in the structure of the set of fixed points of a $1$-parameter family of systems when varying this parameter. Critical points of the Hilbert-Einstein functional in a conformal class $[g]$ are fixed points of the so-called \emph{Yamabe flow}, the corresponding $L^2$-gradient flow of the Hilbert-Einstein functional, which gives a dynamical system in this conformal class. The bifurcation results above mentioned can be hence interpreted as a local change in the set of fixed points of the Yamabe flow near homogeneous metrics (which are always fixed points) when varying the conformal class $[g]$ with $g$ in one of the families $\gt,\hs$ and $\kt$. An interesting question would be to study the dynamics near these new fixed points for $t\in\{t_q^{\h}\}$ (respectively $t\in\{t_q^{\k}\}$).

We also obtain corollaries to our main results regarding \emph{global} uniqueness and multiplicity of solutions to the Yamabe problem in the conformal classes of $\gt,\hs$ and $\kt$. Namely, we prove that the set of $t>0$ such that there exists a \emph{unique} unit volume constant scalar curvature metric in the conformal class of $\gt$ is open (see Proposition~\ref{prop:stability}). The same is true for the other families $\hs$ and $\kt$. In the opposite direction, we have results on the multiplicity of solutions in the conformal classes of the families $\hs$ and $\kt$ that admit bifurcation. More precisely, we prove existence of \emph{at least $3$} different unit volume constant scalar curvature metrics in the conformal class of $\hs$ for $t$ in an infinite subset of $]0,1[$ with $0$ in its closure, see Proposition~\ref{prop:mult}. Again, the same holds for $\kt$.

The techniques to prove all of the above results are based on the variational characterization of solutions to the Yamabe problem as critical points of the Hilbert-Einstein functional on conformal classes. Most of the abstract results used appeared in the recent paper~\cite{lpz}. There are two such abstract results. The first (Proposition~\ref{prop:localrigidity}) gives a sufficient condition for local rigidity at a critical point using the Implicit Function Theorem. The second (Proposition~\ref{prop:bifmorseindex}) gives sufficient conditions for existence of a bifurcation, based on the classic bifurcation result that a change in the Morse index implies bifurcation, see \cite[Thm II.7.3]{kielhofer} or \cite[Thm 2.1]{smwas}.

Thus, in order to prove the results claimed, we have to analyze the second variation of the Hilbert-Einstein functional at every metric\footnote{We denote an abstract $1$-parameter family of metrics by $g_t$ (so that in the applications mentioned $g_t$ will be replaced by $\gt,\hs$ and $\kt$).} $g_t$ of the $1$-parameter families of homogeneous metrics on $S^m$. Using the second variation formula \eqref{eq:jacobiop}, this amounts to computing the spectrum of the Laplacian $\Delta_t$ of $g_t$ and the scalar curvature $\scal(g_t)$.

More precisely, a critical point $g_t$ is degenerate if and only if $\scal(g_t)\neq0$ and $\tfrac{\scal(g_t)}{m-1}$ is an eigenvalue of $\Delta_t$; and the Morse index of a critical point $g_t$ is the number of positive eigenvalues of $\Delta_t$ that are less than $\tfrac{\scal(g_t)}{m-1}$. Since the above $1$-parameter families $g_t$ of homogeneous metrics on $S^m$ are the obtained by scaling the fibers of Riemannian submersions with totally geodesic fibers, the spectrum of the Laplacian $\Delta_t$ of $g_t$ is well-understood (see~\cite{bbb,bb}). Roughly, it consists of linear combinations (that depend on $t$) of eigenvalues of the original metric on the total space, the round sphere, with eigenvalues of the fibers, which are also round spheres. Since the spectrum of round spheres is well-known, computing the spectrum of $\Delta_t$ is reduced to identifying which combinations of eigenvalues occur as eigenvalues of $\Delta_t$. This depends heavily on the global geometry of the submersion, but, in the cases studied, can be tackled using observations of Tanno~\cite{tanno,tanno2}. In particular, we explicitly compute the first positive eigenvalue $\lambda_1(t)$ of $g_t$. Combining this knowledge of the spectrum of $\Delta_t$ with the formula for the scalar curvature of $g_t$, we are able to identify all degeneracy values of each $1$-parameter family $g_t$, compute their the Morse index and prove existence of bifurcation at all degeneracy values $t\neq 1$ and local rigidity around $g_t$ for all other $t>0$.

We stress that although we only consider spheres, many of the theoretic aspects of our results are completely general to metrics $g_t$ obtained by scaling the fibers of any Riemannian submersion with totally geodesic fibers. With a few other assumptions, it is possible to obtain similar bifurcation results on other such families whose fiber has positive scalar curvature. An infinite sequence of bifurcation values $\{t_q\}$ converging to zero is obtained via changes of the Morse index when eigenvalues of the Laplacian on the base manifold (that are also eigenvalues of $\Delta_t$ for all $t>0$) become smaller than $\tfrac{\scal(g_t)}{m-1}$. Nevertheless, one has to deal with \emph{compensations} of eigenvalues crossing in the opposite direction with same multiplicity (preventing the Morse index from changing). A study of this compensation problem and applications are to appear in a forthcoming note by the authors.

This paper is organized as follows. In Section~\ref{sec:var}, we briefly recall the variational characterization of the Yamabe problem, and establish the basic notions of bifurcation and local rigidity, as well as the abstract bifurcation criteria that will be used in the applications. The Laplacians of a Riemannian submersion with totally geodesic fibers are studied in Section~\ref{sec:subm}. Section~\ref{sec:homsph} describes the classification of homogeneous metrics on spheres, obtained by Ziller~\cite{ziller}. In Sections~\ref{sec:s2n+1t}, \ref{sec:s4n+3t} and \ref{sec:s15t}, we prove the main results of the paper, on bifurcation and local rigidity of solutions to the Yamabe problem, for the families $(S^{2n+1},\gt)$, $(S^{4n+3},\hs)$ and $(S^{15},\kt)$ respectively. Finally, Section~\ref{sec:final} contains global uniqueness and multiplicity results for solutions of the Yamabe problem on the above families that follow as consequences of our main results.

\noindent
{\bf Acknowledgment.} It is our pleasure to thank Karsten Grove, Matthew Gursky, Brian Hall and Wolfgang Ziller for enlightening conversations and for bringing some references to our attention.

\section{Variational setup for the Yamabe problem}
\label{sec:var}

We start by briefly recalling the classic variational setup for the Yamabe problem, see~\cite{lpz,Schoen87} for details. Let $M$ be a closed manifold of dimension $m$. Henceforth we fix $k\geq 3$ and $\alpha\in\,]0,1[$. Consider the set $\Met(M)$ of $C^{k,\alpha}$ Riemannian metrics on $M$, which is an open convex cone in the Banach space of $C^{k,\alpha}$ symmetric $(0,2)$-tensors. For each $g\in\Met(M)$, define the $C^{k,\alpha}$ conformal class of $g$ as
\begin{equation*}
[g]:=\left\{\phi\,g:\phi\in C^{k,\alpha}(M),\phi>0\right\}.
\end{equation*}
Denote by $\Met_1(M)$ the smooth codimension $1$ embedded submanifold of $\Met(M)$ formed by unit volume metrics. Finally, let
\begin{equation*}
[g]_1:=\Met_1(M)\cap [g].
\end{equation*}
The set $[g]_1$ is a smooth codimension $1$ Banach submanifold of $[g]$, and its tangent space at the metric $g$ can be canonically identified as
\begin{equation}\label{eq:idtmet}
T_{g}[g]_1\cong\left\{\psi\in C^{k,\alpha}(M):\int_M\psi\; \vol_{g}=0\right\}.
\end{equation}

\subsection{Hilbert-Einsten functional}\label{subs:var}
The Hilbert-Einstein functional is defined by
\begin{equation}\label{eq:a}
\mathcal A(g):=\frac{1}{\Vol(g)}\int_M \scal(g)\; \vol_g.
\end{equation}

\begin{proposition}
The Hilbert-Einstein functional $\mathcal A$ is smooth on $[g]_1$, and a metric $g'\in [g]_1$ is a critical point of $\mathcal A|_{[g]_1}$ if and only if it has constant scalar curvature. If $g\in [g]_1$ is a critical point, the second variation can be identified with the quadratic form on \eqref{eq:idtmet} given by
\begin{equation}\label{eq:jacobiop}
\dd^2\big(\mathcal A|_{[g]_1}\big)(g)(\psi,\psi)=\tfrac{m-2}{2}\int_M\big((m-1)\Delta_{g}\psi-\scal(g)\psi\big)\psi\;\vol_{g}.
\end{equation}
\end{proposition}

\begin{remark}\label{rem:renorm}
Note that, given $\alpha>0$, one has $\Delta_{\alpha g}=\frac1\alpha\Delta_g$ and $\scal(\alpha g)=\frac1\alpha\scal(g)$. Hence, the spectrum of $\Delta_g-\frac{\scal(g)}{m-1}$ scales in a trivial way under homotheties, in the sense that negative (respectively positive) eigenvalues remain negative (respectively positive). On the other hand, $\vol_{\alpha g}=\alpha^\frac m2\vol_g$. Thus, whenever necessary, we may renormalize a metric to have unit volume without compromising the above spectral theory.
\end{remark}

\subsection{Bifurcation vs. local rigidity}
Consider a smooth path $g_t\in\Met_1(M)$ of constant scalar curvature Riemannian metrics on $M$.

\begin{definition}\label{def:biflocrig}
A value $t_*\in[a,b]$ is a \emph{bifurcation value} for the family $g_t$ if there exists a sequence $\{t_q\}$ in $[a,b]$ that converges to $t_*$ and a sequence $\{g_q\}$ in $\Met(M)$ of Riemannian metrics on $M$ that converges to $g_{t_*}$ satisfying for all $q\in\N$
\begin{itemize}
\item[(i)] $g_{t_q}\in [g_q]$, but $g_q\neq g_{t_q}$;
\item[(ii)] $\Vol(g_q)=\Vol(g_{t_q})$;
\item[(iii)] $\scal(g_q)$ is constant.
\end{itemize}

If $t_*\in[a,b]$ is not a bifurcation value, then the family $g_t$ is \emph{locally rigid} at $t_*$. In other words, the family $g_t$ is locally rigid at $t_*\in[a,b]$ if there exists an open neighborhood $U$ of $g_{t_*}$ in $\Met(M)$ such that the following holds. If $g\in U$ is another metric with $\scal(g)$ constant and there exists $t\in [a,b]$ with $g_t\in U$, $g\in [g_{t}]$ and $\Vol(g)=\Vol(g_{t})$, then $g=g_{t}$.
\end{definition}

A value $t_*\in[a,b]$ is a \emph{degeneracy value} for $g_t$ if $\scal({g_{t_*}})\neq0$ and $\tfrac{\scal({g_{t_*}})}{m-1}$ is an eigenvalue of the Laplacian $\Delta_{t_*}$ of $g_{t_*}$. From \eqref{eq:jacobiop}, this is equivalent to $g_{t_*}$ being a \emph{degenerate} critical point of \eqref{eq:a} restricted to $[g_{t_*}]_1$. The \emph{Morse index} $N(g_t)$ of $g_t$ is the number of positive eigenvalues of $\Delta_t$ (counted with multiplicity) that are less than $\tfrac{\scal({g_t})}{m-1}$.

The following sufficient condition for local rigidity of $g_t$ at $t_*$ is easily proved using the Implicit Function Theorem, see \cite[Prop 3.1]{lpz}.

\begin{proposition}\label{prop:localrigidity}
If $t_*$ is not a degeneracy value of $g_t$, then $g_t$ is locally rigid at $t_*$.
\end{proposition}

\begin{corollary}\label{cor:indzero}
Suppose that, in addition to the hypotheses of Proposition~\ref{prop:localrigidity}, there exists a value $t_*$ for which $\frac{\scal(g_{t_*})}{m-1}$ is less than the first positive eigenvalue $\lambda_1(t_*)$ of $\Delta_{t_*}$. Then $g_{t_*}$ is a \emph{strict local minimum} for the Hilbert-Einstein functional in its conformal class, in particular $g_t$ is locally rigid at $t_*$.
\end{corollary}

Thus, a necessary condition for bifurcation at $t_*$ is that $t_*$ be a degeneracy value. However, this is generally not sufficient. From classic bifurcation theory, a \emph{sufficient} condition can be given in terms of a change in the Morse index $N(g_t)$ when passing a degeneracy value $t_*$, see \cite[Thm II.7.3]{kielhofer} or \cite[Thm 2.1]{smwas}. This provides our main tool in detecting bifurcation:

\begin{proposition}\label{prop:bifmorseindex}
Assume $a$ and $b$ are not degeneracy values for $g_t$ and $N(g_a)\neq N(g_b)$. Then there exists a bifurcation value $t_*\in\,]a,b[$ for the family $g_t$.
\end{proposition}

\begin{proof}
See  \cite[Thm 3.3]{lpz}.
\end{proof}

\section{Spectrum of Riemannian submersions with totally geodesic fibers}
\label{sec:subm}

Great part of the proofs of our results on homogeneous spheres are based in general facts for Riemannian submersions with totally geodesic fibers. In this section, we recall abstract tools to study their spectrum. Let $\pi\colon (M,g)\to (B,h)$ be a smooth Riemannian submersion with (connected) totally geodesic fibers $(F,k)$.

\subsection{Laplacians}
Let $\Delta_M$ be the Laplacian of $(M,g)$, acting on $C^\infty(M)$. Then $\Delta_M$, as a densely defined operator in $L^2(M,\vol_g)$, is symmetric (hence closable) and non-negative. Furthermore, it is well-known that $\Delta_M$ is essentially self-adjoint with this domain. We denote its unique self-adjoint extension also by $\Delta_M$. Analogously, let $\Delta_F$ be the (unique self-adjoint extension of the) Laplacian of the fiber $(F,k)$.

\begin{definition}\label{def:horverlapl}
Define the \emph{vertical Laplacian} $\Delta_v$ acting on $L^2(M,\vol_g)$ by
\begin{equation*}
(\Delta_v \psi)(p):=(\Delta_F \psi|_{F_p})(p),
\end{equation*}
and the \emph{horizontal Laplacian} acting on the same space by the difference $\Delta_h=\Delta_M-\Delta_v$.
\end{definition}

Both $\Delta_h$ and $\Delta_v$ are non-negative self-adjoint unbounded operators on the Hilbert space $L^2(M,\vol_g)$, but are in general not elliptic (unless $\pi$ is a covering). We now consider the spectra of such operators. It is well-known that $\Delta_M$ is non-negative and has compact resolvent, hence its spectrum is non-negative and discrete. Since the fibers are isometric, $\Delta_v$ also has discrete spectrum equal to that of the fibers. Let us denote these spectra by
\begin{equation}
\begin{aligned}\label{spectra}
\sigma(\Delta_M) &= \{0=\mu_0<\mu_1<\cdots<\mu_k\nearrow +\infty\}\\
\sigma(\Delta_v) &= \{0=\phi_0<\phi_1<\cdots<\phi_j\nearrow +\infty\}.
\end{aligned}
\end{equation}
We stress that the multiplicity of the eigenvalues of $\Delta_M$ is always \emph{finite}, however the eigenvalues of $\Delta_v$ might have \emph{infinite} multiplicity. Namely, $\Delta_v \psi=0$ only implies that $\psi$ is constant \emph{along the fibers}, i.e., $\psi=\overline \psi\circ\pi$ for some function $\overline \psi$ on the base. We also remark that the spectrum of $\Delta_h$ \emph{need not be discrete} (see \cite[Example 3.4]{bbb}), and it contains but does not coincide with the spectrum of the base $B$.

Having totally geodesic fibers guarantees the following key property of the operators $\Delta_M$ and $\Delta_v$ of a submersion, see~\cite[Thm 3.6]{bbb}.

\begin{theorem}\label{thm:bbb}
The Hilbert space $L^2(M,\vol_g)$ admits a Hilbert basis consisting of simultaneous eigenfunctions of $\Delta_M$ and $\Delta_v$.
\end{theorem}

\begin{remark}\label{rem:comm}
The proof of the above result in~\cite[Thm 3.6]{bbb} is slightly incomplete. Namely, the basis of simultaneous eigenfunctions is claimed to be obtained from \emph{commutativity} of $\Delta_M$ and $\Delta_v$. However, it is only proven that the unbounded self-adjoint operators $\Delta_M$ and $\Delta_v$ commute \emph{on the dense subspace} of smooth functions on $M$, which is not sufficient\footnote{In fact, counter-examples to this implication were given by Nelson, see~\cite[Chap VIII]{rs}.} to draw the above conclusion. The correct notion of \emph{commutativity} for such operators in order to obtain this conclusion is in terms of commutativity of all of their projections in the associated projection-valued measures. This strong commutativity can be proved using~\cite[Cor 9.2]{nelson} with $A=\Delta_M$, $B=\Delta_v$ and $\mathfrak D=C^\infty(M)$. We note that essential self-adjointness of $A^2+B^2$ on $\mathfrak D$ follows from the fact that this is an elliptic operator with smooth coefficients in a compact manifold. An alternative proof of the desired strong commutativity can be given using~\cite[Prop 2]{schmudgen} with $A=B_1=\Delta_M$, $B_2=\Delta_v$ and $\mathcal D_1=\mathcal D_2=\mathcal D_{12}=C^\infty(M)$.
\end{remark}

A thorough description of the spectrum of Riemannian submersions with totally geodesic fibers was later achieved in \cite{bb}, in terms of representations of the isometry group of the fibers.

\subsection{Scaling the fibers}
Let $F\to (M,g)\stackrel{\pi}{\to} B$ be a Riemannian submersion with totally geodesic fibers. A natural way of obtaining a $1$-parameter family of other such submersions is scaling the original metric in the direction of the fibers. Namely, consider $g_t\in\Met(M)$ given by
\begin{equation}\label{eq:gt-geral}
g_t(v,w):=\begin{cases} t^2g(v,w), & v,w \mbox{ vertical}\\ 0, & v \mbox{ vertical}, \; w \mbox{ horizontal }\\ g(v,w), & v,w \mbox{ horizontal.}\\ \end{cases}
\end{equation}
Then $F\to (M,g_t)\stackrel{\pi}{\to} B$, $t>0$, is a Riemannian submersion with totally geodesic fibers, isometric to $(F,t^2k)$, where $(F,k)$ is the original fiber of $\pi\colon M\to B$. This construction is sometimes referred to as \emph{canonical variation}. Note that for $a\neq b$, the metrics $g_a$ and $g_b$ are \emph{not} conformal.

\begin{proposition}
Let $\Delta_t$ be the Laplacian of $(M,g_t)$. Then
\begin{equation}\label{eq:laplacians}
\Delta_t=\Delta_h+\tfrac{1}{t^2}\Delta_v=\Delta_M+\left(\tfrac{1}{t^2}-1\right)\Delta_v.
\end{equation}
\end{proposition}

For the proof of the above formula, see \cite[Prop 5.3]{bbb}. We will now use it to characterize the spectrum of $\Delta_t$ in terms of the spectrum of the original Laplacian $\Delta_M$ and the vertical Laplacian $\Delta_v$.

\begin{corollary}\label{cor:specincl}
For each $t>0$, the following inclusion holds
\begin{equation*}
\sigma(\Delta_t)\subset\sigma(\Delta_M)+\left(\tfrac{1}{t^2}-1\right)\sigma(\Delta_v).
\end{equation*}
Since the above spectra are discrete, this means that every eigenvalue $\lambda(t)$ of $\Delta_t$ is of the form
\begin{equation}\label{eq:lambdakj}
\lambda^{k,j}(t)=\mu_k+\left(\tfrac{1}{t^2}-1\right)\phi_j,
\end{equation}
for some eigenvalues $\mu_k$ and $\phi_j$ of $\Delta_M$ and $\Delta_v$ respectively.
\end{corollary}

\begin{proof}
It follows from Theorem~\ref{thm:bbb} that $\Delta_M$ and $\Delta_v$ \emph{commute}.\footnote{Since these are \emph{unbounded} self-adjoint operators, the correct notion of commutativity is given in terms of commutativity of the projections in the associated projection-valued measures, see Remark~\ref{rem:comm}.} Thus, by the Spectral Theorem, such operators are simultaneously diagonalizable in the sense that there exists a unitary operator $U$ of $L^2(M,\vol_g)$ such that $U\Delta_M U^{-1}=T_{f_M}$ and $U\Delta_v U^{-1}=T_{f_v}$ are operators of multiplication by functions $f_M$ and $f_v$ respectively. For such multiplication operators $T_f$, the spectrum $\sigma(T_f)$ is the essential range of $f$. Thus, from \eqref{eq:laplacians}, $\sigma(\Delta_t)$ is the essential range of $f_M+\left(\tfrac{1}{t^2}-1\right)f_v$, which is contained in the closure of the sum of the essential ranges of $f_M$ and $\left(\tfrac{1}{t^2}-1\right)f_v$, which in turn is the closure of $\sigma(\Delta_M)+\left(\tfrac{1}{t^2}-1\right)\sigma(\Delta_v)$. Since both spectra are discrete, we may remove the closure and the proof is complete.
\end{proof}

Not all possible combinations of $\mu_k$ and $\phi_j$ on \eqref{eq:lambdakj} give rise to an eigenvalue of $\Delta_t$. In fact, this only happens when the total space of the submersion is a Riemannian product. As mentioned before, determining which combinations are allowed depends on the global geometry of the submersion. In the general case, this problem was solved in \cite{bb}. As we will see, in the case of homogeneous spheres, one can use simpler observations of Tanno~\cite{tanno,tanno2} to obtain a sufficiently good refinement of Corollary~\ref{cor:specincl}, see Lemmas~\ref{lemma:tanno4n+3} and \ref{lemma:tanno15}. From the above, we obtain a Hilbert basis of eigenfunctions of $\Delta_t$ which clearly varies with $t$. In particular, the ordering of the eigenvalues of $\Delta_t$ \emph{may} change with $t$. This behavior indeed occurs in our applications.

\begin{remark}
The canonical variation $g_t$ only exists for $0<t<+\infty$, however, considering spectral theory also on \emph{Alexandrov spaces}, one may inquire about the spectrum of the limits $t\to 0$ and $t\to +\infty$. Indeed, the Gromov-Hausdorff limits of $(M,g_t)$ as $t\to 0$ and as $t\to +\infty$ exist. The first is the base $B$ of the submersion, since the fibers collapse to a point. The second is a compact sub-Riemannian manifold (typically with Hausdorff dimension higher than $m$). The $k^{th}$ eigenvalue $\lambda_k(t)$ of the Laplacian $\Delta_t$ is a continuous (and, in general, not differentiable) function of $t$, for $t\in[0,+\infty]$, see \cite{shioya}. In particular, the limits $\lim_{t\to0}\lambda_k(t)$ and $\lim_{t\to+\infty}\lambda_k(t)$ are eigenvalues of the Laplacian of $B=\lim_{t\to0}(M,g_t)$ and of $\lim_{t\to+\infty}(M,g_t)$, respectively. The above claims will be indeed verified in our examples.
\end{remark}

\section{Homogeneous metrics on spheres}
\label{sec:homsph}

Consider the following Riemannian submersions with totally geodesic fibers
\begin{equation*}
S^1\to S^{2n+1}\to\C P^n, \quad S^3\to S^{4n+3}\to\Hr P^n, \quad S^7\to S^{15}\to S^8\left(\tfrac12\right),
\end{equation*}
where the fibers and total spaces are equipped with the round metric. Denote the canonical variation of the above by
\begin{equation}\label{eq:hommetrics}
(S^{2n+1},\gt), \quad (S^{4n+3},\hs), \quad (S^{15},\kt).
\end{equation}
These are $1$-parameter families of homogeneous metrics (often called \emph{Berger spheres}) that are invariant under $\U(n+1)$, $\Sp(n+1)$ and $\Spin(9)$, respectively. In the remainder of this paper, we study bifurcation of solutions to the Yamabe problem along one of the three above families.

Homogeneous metrics on spheres (actually, on compact symmetric spaces of rank one) were classified by Ziller~\cite{ziller} in 1982. Apart from $\{\gt\}$, $\{\hs\}$ and $\{\kt\}$, the only remaining homogeneous metrics on spheres are the $\Sp(n+1)$-homogeneous metrics on $S^{4n+3}$ obtained from scaling the three different circles in the fiber $S^3$ by \emph{distinct} factors (the metrics $\{\hs\}$ correspond to scaling all three circles the same amount). In other words, those remaining metrics are the ones for which the $3$-dimensional Hopf fibers of $S^3\to S^{4n+3}\to \Hr P^n$ are endowed with a left-invariant metric on $S^3$ that is \emph{not} a multiple of the round metric. Note that any two different homogeneous metrics on $S^m$ are non-conformal.

\begin{remark}
It would be possible to use our same techniques to give statements about this $3$-parameter family by first considering the deformation $\hs$ and then deforming again by scaling the metric along one of the three possible subgroups $S^1$ inside the fiber $S^3$. Nevertheless, this yields somewhat artificial results since only deformations in very specific directions inside this family are considered.
\end{remark}

\subsection{Scalar curvatures}
The metrics $\{\gt\}$, $\{\hs\}$ and $\{\kt\}$ are homogeneous hence have constant scalar curvature. In order to compute them, one can use the Gray-O'Neill formulas, obtaining the following (cf.~\cite[Examples 9.81, 9.82 and 9.84]{besse}).

\begin{proposition}\label{prop:scalberger}
The following formulas hold:
\begin{eqnarray}
\scal(S^{2n+1},\gt) &=& 2n(2n+2-t^2),\label{eq:scalberger} \\
\scal(S^{4n+3},\h_t) &=& 2\left(\tfrac{3}{t^2}+8n(n+2)-6nt^2\right), \label{eq:scal4n+3t} \\
\scal(S^{15},\k_t) &=& 14\left(\tfrac{3}{t^{2}} + 16 - 4 t^2\right). \label{eq:scal15t}
\end{eqnarray}
\end{proposition}

\section{\texorpdfstring{Homogeneous spheres $(S^{2n+1},\gt)$}{Homogeneous spheres I}}\label{sec:s2n+1t}

\subsection{First eigenvalue of the Laplacian}
We use Corollary~\ref{cor:specincl} to study the eigenvalues of the Laplacian $\Delta_t$ of $(S^{2n+1},\g_t)$. As mentioned before, the spectrum of $\Delta_v$ coincides with that of the fibers. The $k^{th}$ eigenvalue\footnote{Henceforth, by an eigenvalue of a Riemannian manifold $(M,g)$ we mean an eigenvalue of its Laplacian $\Delta_g$ acting on $L^2(M,\vol_g)$.} of a $m$-dimensional round sphere is well-known to be $k(k+m-1)$, see \cite{bgm}. Thus, from \eqref{eq:lambdakj}, the eigenvalues of $\Delta_t$ are of the form
\begin{equation*}
\lambda^{k,j}(t)=k(k+2n)+\left(\tfrac{1}{t^2}-1\right)j^2,
\end{equation*}
for some $k,j\in\N\cup\{0\}$. The following refinement of the above is due to Tanno~\cite[Lemma 3.1]{tanno}.

\begin{lemma}\label{lemma:tanno}
For each eigenvalue $\mu_k=k(k+2n)$ of $(S^{2n+1},\g_1)$, denote its corresponding eigenspace by $E_k\subset L^2(S^{2n+1},\vol_{\g_1})$. Then $E_k$ has a orthogonal decomposition\footnote{Notice that some of the $E_k^j$'s may be trivial.} in simultaneous eigenspaces
\begin{equation*}
E_k=E_k^k+E_k^{k-2}+\dots+E_k^{k-2\left\lfloor k/2\right\rfloor},
\end{equation*}
where $\left\lfloor k/2 \right\rfloor$ denotes the largest integer less than or equal to $k/2$, and for each $\psi\in E_k^j$, $\Delta_v \psi=j^2 \psi$. In particular, $\lambda^{k,j}(t)$ can only be an eigenvalue of $\Delta_t$ if $0\leq j\leq k$ and $k-j$ is even, i.e., the eigenvalues of $\Delta_t$ can only be among
\begin{equation}\label{eq:lambdas}
\lambda^{k,j}(t)=k(k+2n)+\left(\tfrac{1}{t^2}-1\right)j^2, \quad j=k,k-2,\dots,k-2\left\lfloor k/2 \right\rfloor.
\end{equation}
\end{lemma}

\begin{proof}
Note that since the fibers have dimension one, denoting by $V$ the unit tangent field to the Hopf fibers and identifying it with the derivation operator of $L^2(S^{2n+1},\vol_{\g_1})$ in the direction $V$, we have $\Delta_v=-VV$. So, if $\psi\in E_k$ satisfies $\Delta_v\psi=j^2\psi$, then along any geodesic $\gamma(s)$ tangent to $V$, we must have
\begin{equation}\label{eq:res1}
\psi(\gamma(s))=c_1\sin(js+c_2),
\end{equation}
where $c_1$ and $c_2$ are real constants depending on $\gamma$.

On the other hand, since $\psi\in E_k$, from standard theory of spherical harmonics, $\psi$ is the restriction of a harmonic homogeneous polynomial of degree $k$ in $\R^{2n+2}$ to $S^{2n+1}$, see \cite{bgm}. Thus, the restriction of $\psi$ to $\gamma$ must also be of the form
\begin{equation}\label{eq:res2}
\psi(\gamma(s))=\sum_{l=0}^k c_l\cos^l(s)\sin^{k-l}(s),
\end{equation}
where $c_l$ are real constants depending on $\gamma$. Since \eqref{eq:res1} and \eqref{eq:res2} must coincide, it follows that $j\in\{k,k-2,\dots,k-2\left\lfloor k/2 \right\rfloor\}$.
\end{proof}

This already imposes several necessary conditions on $k,j\in\N\cup\{0\}$ to form an eigenvalue $\lambda^{k,j}(t)$ of $\Delta_t$. Clearly, the problem of determining which such $k,j\in\N\cup\{0\}$ indeed give rise to $\lambda^{k,j}(t)\in\sigma(\Delta_t)$ is equivalent to determining which $E_k^j$ are non-trivial. By constructing explicit elements of $E_k^j$ using spherical harmonics, Tanno~\cite{tanno} also observed the following.

\begin{lemma}\label{lemma:tanno2}
The following $E_k^j$ are non-trivial:
\begin{itemize}
\item[(i)] $E_k^k$, for any $k$;
\item[(ii)] $E_k^1$, if $k$ is odd;
\item[(iii)] $E_k^0$ and $E_k^2$, if $k$ is even.
\end{itemize}
\end{lemma}

We are now ready to compute the bottom of the spectrum of $(S^{2n+1}\g_t)$, by extracting the least possible $\lambda^{k,j}(t)$ among \eqref{eq:lambdas}.

\begin{proposition}\label{prop:firsteig}
The first non-zero eigenvalue of the Laplacian of $(S^{2n+1},\g_t)$ is
\begin{equation}\label{eq:lambda1s2n+1}
\lambda_1(t)=\begin{cases} 4(1+n), & 0<t\leq\tfrac{1}{\sqrt{2(2+n)}} \\ 2n+\tfrac{1}{t^2}, & t\geq\tfrac{1}{\sqrt{2(2+n)}} \end{cases}
\end{equation}
with multiplicity given by
\begin{equation*}
m_1(t)=\begin{cases} n(n+2), & 0<t<\tfrac{1}{\sqrt{2(2+n)}} \\
n^2+4n+2, & t=\tfrac{1}{\sqrt{2(2+n)}} \\
2(n+1), & t>\tfrac{1}{\sqrt{2(2+n)}}.\end{cases}
\end{equation*}
\end{proposition}

\begin{proof}
From Lemma~\ref{lemma:tanno}, we know that any eigenvalues of $\Delta_t$ are of the form $\lambda^{k,j}(t)$ for some $k,j\in\N\cup\{0\}$, as in \eqref{eq:lambdas}. We want to compute
\begin{eqnarray}\label{eq:l1r}
\lambda_1(t) &=& \min\{ \lambda^{k,j}(t)\in\sigma(\Delta_t): k,j\in\N\cup\{0\}, j\leq k, k^2+j^2\neq0\}\nonumber \\
&=& \min\{r_1(t),r_2(t)\},
\end{eqnarray}
where
\begin{eqnarray*}
r_1(t) &=& \min\{\lambda^{k,0}(t)\in\sigma(\Delta_t): k\in\N\},\\
r_2(t) &=& \min\{\lambda^{k,j}(t)\in\sigma(\Delta_t):k,j\in\N,1\leq j\leq k\}.
\end{eqnarray*}
First, suppose $1\leq j\leq k$, i.e., compute $r_2(t)$. For such $k,j$,
\begin{equation*}
\lambda^{k,j}(t) = k(k+2n)+\left(\tfrac{1}{t^2}-1\right)j^2 = k^2-j^2+2n+\tfrac{j^2}{t^2}\geq 2n+\tfrac{1}{t^2} =\lambda^{1,1}(t).
\end{equation*}
From Lemma~\ref{lemma:tanno2}, $E_1^1$ is non-trivial, hence $\lambda^{1,1}(t)\in\sigma(\Delta_t)$. This proves that
\begin{equation}\label{eq:r2}
r_2(t)=\lambda^{1,1}(t)=2n+\tfrac{1}{t^2}.
\end{equation}

Now, suppose $j=0$. From Lemma~\ref{lemma:tanno}, if $\lambda^{k,j}(t)\in\sigma(\Delta_t)$, then $k\in\N$ must be even. Thus,
\begin{equation*}
\lambda^{k,0}(t) = k(k+2n) \geq2(2+2n) =\lambda^{2,0}(t).
\end{equation*}
From Lemma~\ref{lemma:tanno2}, $E_2^0$ is non-trivial, hence $\lambda^{2,0}(t)\in\sigma(\Delta_t)$. This proves that
\begin{equation}\label{eq:r1}
r_1(t)=\lambda^{2,0}(t)=4(1+n).
\end{equation}

From \eqref{eq:l1r}, \eqref{eq:r2} and \eqref{eq:r1}, we obtain:
\begin{equation*}
\lambda_1(t)=\min\left\{4(1+n),2n+\tfrac{1}{t^2}\right\}=\begin{cases} 4(1+n), & 0<t\leq\tfrac{1}{\sqrt{2(2+n)}} \\ 2n+\tfrac{1}{t^2}, & t\geq\tfrac{1}{\sqrt{2(2+n)}}.\end{cases}
\end{equation*}

To compute the multiplicity of this first eigenvalue, one has to compute dimensions of $E_k^j$. These can be determined using the formula for the dimensions of $E_k$, see \cite{bgm}. For this computation, we refer to~\cite{tanno}.
\end{proof}

\subsection{Local rigidity}
We now prove local rigidity of the family $(S^{2n+1},\g_t)$ at all $t>0,t\neq1$, in the sense of Definition~\ref{def:biflocrig}. Recall this family is \emph{not} locally rigid at $t=1$, which is a known bifurcation value, since the critical points of the Hilbert-Einstein functional restricted to the conformal class $[\g_1]$ of the round metric form a manifold of dimension $2n+2$.

\begin{theorem}\label{thm:s2n+1}
The family $(S^{2n+1},\g_t)$ is locally rigid at all $t>0,t\neq1$.
\end{theorem}

\begin{proof}
Recall formulas \eqref{eq:scalberger} from Proposition~\ref{prop:scalberger} and \eqref{eq:lambda1s2n+1} from Proposition~\ref{prop:firsteig}. A simple analysis shows that for $t\leq\tfrac{1}{\sqrt{2(2+n)}}$, $\tfrac{1}{2n}\scal(S^{2n+1},\g_t)<\lambda_1(t)$. For $t\geq\tfrac{1}{\sqrt{2(2+n)}}$, we also have that
\begin{equation*}
\tfrac{1}{2n}\scal(S^{2n+1},\g_t)=2n+2-t^2<2n+\tfrac{1}{t^2}=\lambda_1(t),
\end{equation*}
unless $t=1$, in which case equality holds. Local rigidity of $(S^{2n+1},\g_t)$ for $t\neq1$ now follows from Corollary~\ref{cor:indzero}.
\end{proof}

\section{\texorpdfstring{Homogeneous spheres $(S^{4n+3},\hs)$}{Homogeneous spheres II}}\label{sec:s4n+3t}

\subsection{Spectrum of the Laplacian}
Analogously to the case before, from \eqref{eq:lambdakj}, all eigenvalues of $\Delta_t$ are of the form
\begin{equation*}
\lambda^{k,j}(t)=k(k+4n+2)+\left(\tfrac{1}{t^2}-1\right)j(j+2),
\end{equation*}
for some $k,j\in\N\cup\{0\}$. The following refinement of the above is due to Tanno~\cite[Prop 3.1]{tanno2}.

\begin{lemma}\label{lemma:tanno4n+3}
For each eigenvalue $\mu_k=k(k+4n+2)$ of $(S^{4n+3},\h_1)$, denote its corresponding eigenspace by $E_k\subset L^2(S^{4n+3},\vol_{\h_1})$. Then $E_k$ has a orthogonal decomposition in simultaneous eigenspaces
\begin{equation*}
E_k=E_k^k+E_k^{k-2}+\dots+E_k^{k-2\left\lfloor k/2\right\rfloor},
\end{equation*}
where $\left\lfloor k/2 \right\rfloor$ denotes the largest integer less than or equal to $k/2$, and for each $\psi\in E_k^j$, $\Delta_v \psi=j(j+2) \psi$. In particular, $\lambda^{k,j}(t)$ can only be an eigenvalue of $\Delta_t$ if $0\leq j\leq k$ and $k-j$ is even, i.e., the eigenvalues of $\Delta_t$ can only be among
\begin{equation}\label{eq:lambdas4n+3}
\lambda^{k,j}(t)=k(k+4n+2)+\left(\tfrac{1}{t^2}-1\right)j(j+2), \quad j=k,k-2,\dots,k-2\left\lfloor k/2 \right\rfloor.
\end{equation}
\end{lemma}

\begin{proof}
The proof is totally analogous to that of Lemma~\ref{lemma:tanno}. Namely, one observes that $\Delta_v=-\sum_{l=1}^3V_lV_l$, where $\{V_l,l=1,2,3\}$ are unit vectors that span the vertical subbundle, i.e., the distribution tangent to the Hopf fibers. Then, a simultaneous eigenfunction $\psi\in E_k^j$ must be the restriction of a harmonic homogeneous polynomial of degree $k$ to $S^{4n+3}\subset\R^{4n+4}$ and hence the degree of the restriction to the Hopf fibers must be one of $k,k-2,\dots,k-2\left\lfloor k/2 \right\rfloor$.
\end{proof}

Similarly to the previous case, we have the following.

\begin{lemma}\label{lemma:tanno24n+3}
The following $E_k^j$ are non-trivial:
\begin{itemize}
\item[(i)] $E_k^k$, for any $k$;
\item[(ii)] $E_k^0$ if $k$ is even.
\end{itemize}
\end{lemma}

As a consequence, we may compute the the first eigenvalue of $(S^{4n+3},\h_t)$, by extracting the least possible $\lambda^{k,j}(t)$ among \eqref{eq:lambdas4n+3} knowing that some of them occur as eigenvalues of $\Delta_t$ for the above $k,j\in\N\cup\{0\}$ such that $E_k^j$ are non-trivial.

\begin{proposition}\label{prop:firsteig4n+3}
The first non-zero eigenvalue of the Laplacian of $(S^{4n+3},\h_t)$ is
\begin{equation*}
\lambda_1(t)=\begin{cases} 8(1+n), & 0<t\leq\tfrac{1}{2\sqrt{2+n}} \\ 4n+\tfrac{3}{t^2}, & t\geq\tfrac{1}{2\sqrt{2+n}} \end{cases}
\end{equation*}
with multiplicity given by
\begin{equation*}
m_1(t)=\begin{cases} n(2n+3), & 0<t<\tfrac{1}{2\sqrt{2+n}} \\
2n^2+7n+4, & t=\tfrac{1}{2\sqrt{2+n}} \\
4(n+1), & t>\tfrac{1}{2\sqrt{2+n}}.\end{cases}
\end{equation*}
\end{proposition}

\subsection{Bifurcation and local rigidity}
Equipped with expressions for the scalar curvature and eigenvalues of the Laplacian of $(S^{4n+3},\h_t)$ we now completely determine the set of values $t>0$ where this family is locally rigid and prove existence of bifurcation in its complementary. In fact, all degeneracy values will be bifurcation values.

\begin{proposition}\label{prop:morseindexh}
The degeneracy values for $\hs$ form a sequence $\{t_q^{\h}\}$, with $t_q^{\h}\to 0$ as $q\to\infty$, given by $t_0^{\h}=1$ and for $q>0$,
\begin{multline}\label{eq:tqh}
t_q^{\h}=\left[\frac23(n-2qn-2q+2)-\frac{q}{3}\left(2q+\frac1n+\frac{q}{n}\right) \right.\\
+\left.\frac{\sqrt{18n+4(4qn^2-2n^2+2q^2n+4qn-4n+q^2+q)^2}}{6n}\right]^{\tfrac12}.
\end{multline}
The Morse index of $\hs$ is given by
\begin{equation}\label{eq:morseh}
N(\hs)=\begin{cases} \sum_{q=1}^r m_q(\Hr P^n), & t^{\h}_{r+1}\leq t<t^{\h}_r\\
0, & t\geq t^{\h}_1,
\end{cases}
\end{equation}
where $m_q(\Hr P^n)$ is the multiplicity of the $q^{th}$ eigenvalue of $\Hr P^n$, see \cite{bgm}. In particular, it changes whenever $t$ crosses a degeneracy value $t_q^{\h}$, is positive for $t<t_1^{\h}$ and gets arbitrarily large for small $t>0$.
\end{proposition}

\begin{proof}
By using elementary calculus techniques and comparing \eqref{eq:scal4n+3t} and \eqref{eq:lambdas4n+3}, we conclude that $\tfrac{1}{4n+2}\scal(S^{4n+3},\h_t)$ can be equal to $\lambda^{k,j}(t)$ if and only if $j=0$, otherwise $\tfrac{1}{4n+2}\scal(S^{4n+3},\h_t)<\lambda^{k,j}(t)$ for any $1\leq j\leq k$ except when $k=j=1$ and $t=1$, as shown in Figure~\ref{fig:bifinst}. Let us give an idea of how to verify these facts.

\begin{figure*}[htf]
\centering
\includegraphics[width=0.6\textwidth]{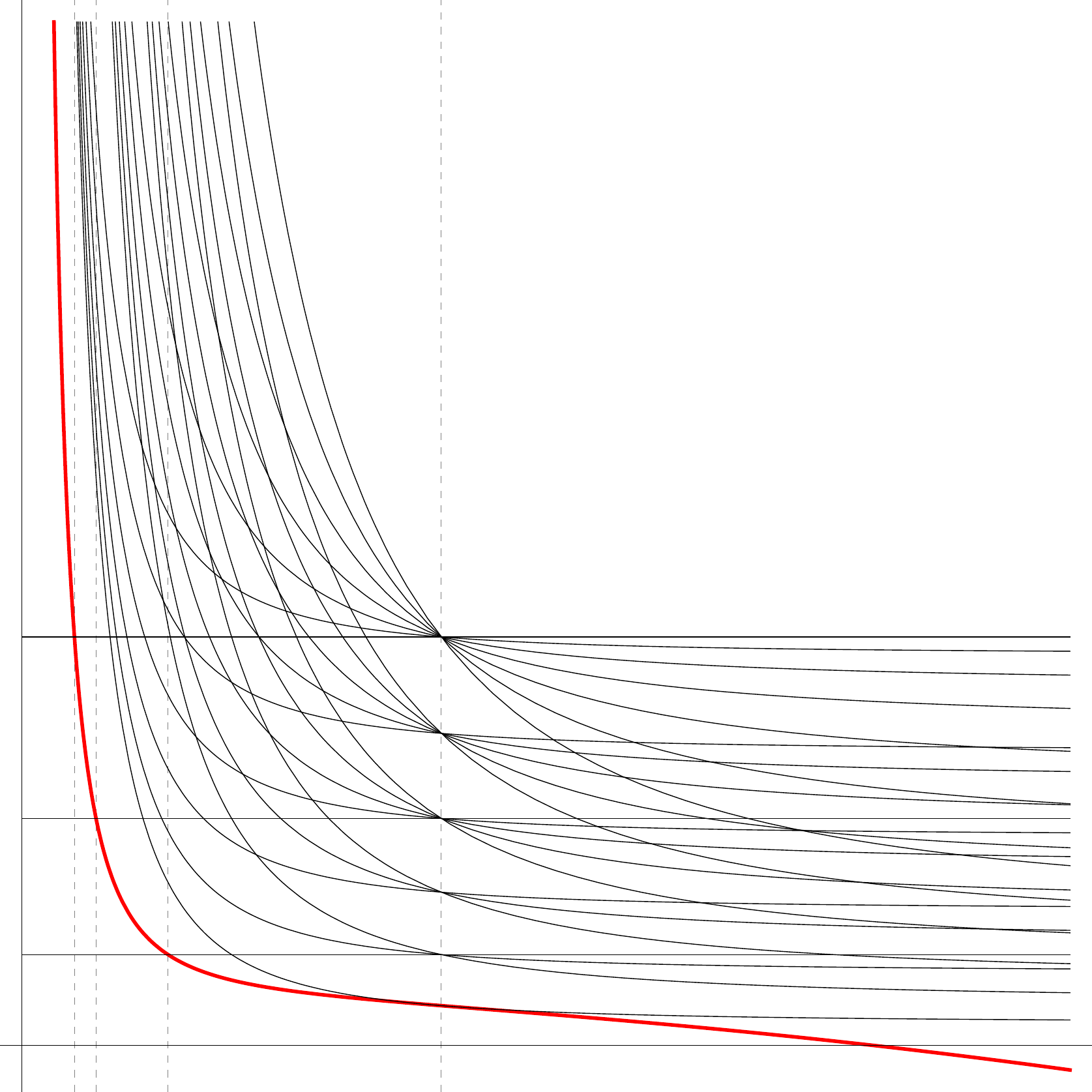}
\vspace{.3cm}
\begin{pgfpicture}
\pgfputat{\pgfxy(-.75,-0.1)}{\pgfbox[center,center]{\color{red}{$s(t)$}}}
\pgfputat{\pgfxy(-4.65,-0.3)}{\pgfbox[center,center]{$t_1^{\h}=1$}}
\pgfputat{\pgfxy(-6.5,-0.3)}{\pgfbox[center,center]{$t_2^{\h}$}}
\pgfputat{\pgfxy(-7.25,-0.3)}{\pgfbox[center,center]{$\dots t_3^{\h}$}}
\pgfputat{\pgfxy(0.1,0.345)}{\pgfbox[center,center]{$t$}}
\pgfputat{\pgfxy(0.4,1.7)}{\pgfbox[center,center]{$\lambda^{k,j}(t)$}}
\end{pgfpicture}
\caption{Graphs of $s(t)=\tfrac{1}{4n+2}\scal(S^{4n+3},\h_t)$ in red and $\lambda^{k,j}(t)$ in black, for $0\leq j\leq k\leq 6$. The dashed vertical lines are over the first degeneracy values (which are all bifurcation values $t_q^{\h}$).}\label{fig:bifinst}
\vspace{-.5cm}
\end{figure*}

Define $\varphi_{n,k,j}\colon \R^+\to\R$ by
\begin{equation*}
\varphi_{n,k,j}(t):=\tfrac{1}{4n+2}\scal(S^{4n+3},\h_t)-\lambda^{k,j}(t), \quad t>0.
\end{equation*}
Consider first the case $1\leq j\leq k$, and note that since $\tfrac{3}{2n+1}<1<j(j+2)$, we have
\begin{equation}\label{eq:lims}
\lim_{t\to0}\varphi_{n,k,j}(t)=-\infty, \quad\quad \lim_{t\to+\infty}\varphi_{n,k,j}(t)=-\infty.
\end{equation}
Moreover, differentiating in $t$ we get
\begin{equation*}
\tfrac{\dd}{\dd t}\varphi_{n,k,j}(t)=\tfrac{1}{t^3}\left(j^2+2j-\tfrac{6}{2n+1}\right)-12nt,
\end{equation*}
so that the function $\varphi_{n,k,j}(t)$ has a unique critical point, namely when $t$ is equal to
\begin{equation*}
c_{n,k,j}=\sqrt[4]{\tfrac{1}{12n}\left(j^2+2j-\tfrac{6}{2n+1}\right)},
\end{equation*}
which in face of \eqref{eq:lims} must be a global maximum for $\varphi_{n,k,j}$. One can then verify that $\varphi_{n,k,j}(c_{n,k,j})\leq0$ and equality holds if and only if $k=j=1$, in which case $c_{n,1,1}=1$. This proves that $\varphi_{n,k,j}(t)\leq0$, i.e., $\tfrac{1}{4n+2}\scal(S^{4n+3},\h_t)\leq\lambda^{k,j}(t)$, for any $1\leq j\leq k$ and equality holds if and only if $k=j=1$, at $t=1$.

Now, let us consider the case $j=0$. In this case,
\begin{equation}\label{eq:lims2}
\lim_{t\to0}\varphi_{n,k,0}(t)=+\infty, \quad\quad \lim_{t\to+\infty}\varphi_{n,k,0}(t)=-\infty
\end{equation}
and the function has negative derivative everywhere, 
\begin{equation*}
\tfrac{\dd}{\dd t}\varphi_{n,k,0}(t)=-\tfrac{6}{2n+1}\tfrac{1}{t^3}-12nt<0.
\end{equation*}
Hence there exists a unique zero of the function $\varphi_{n,k,0}$, which corresponds to a unique $t$ where $\tfrac{1}{4n+2}\scal(S^{4n+3},\h_t)=\lambda^{k,0}(t)$.

Thus, as claimed above, $\tfrac{1}{4n+2}\scal(S^{4n+3},\h_t)$ can be equal to $\lambda^{k,j}(t)$ if and only if $j=0$, and this equality occurs only for one value of $t$, namely the unique zero of $\varphi_{n,k,0}$. If $1\leq j\leq k$, then $\tfrac{1}{4n+2}\scal(S^{4n+3},\h_t)<\lambda^{k,j}(t)$, except when $k=j=1$ and $t=1$, where equality holds.

From Lemmas~\ref{lemma:tanno4n+3} and~\ref{lemma:tanno24n+3}, we know that $\lambda^{k,0}(t)$ is an eigenvalue of $\Delta_t$ if and only if $k$ is even. Setting $k=2q$ for $q\in\N$, we obtain all the degeneracy values of $\hs$ in addition to $t_0^{\h}=1$ as the values $t_q^{\h}$ where $\tfrac{1}{4n+2}\scal(S^{4n+3},\h_t)$ is equal to $\lambda^{2q,0}(t)$, i.e., the values where $\varphi_{n,2q,0}$ attains its only zero. The explicit formula for such values $t_q^{\h}$ is easily verified to be \eqref{eq:tqh}.

Finally, we compute the Morse index of $\hs$. If $t\geq t_1^{\h}$, then there are no positive eigenvalues of $\Delta_t$ that are less than $\tfrac{1}{4n+2}\scal(S^{4n+3},\h_t)$, so the Morse index is zero. When letting $t\to0$, whenever $t$ crosses a degeneracy value $t_q^{\h}$, the eigenvalue $\lambda^{2q,0}(t)$ becomes smaller than $\tfrac{1}{4n+2}\scal(S^{4n+3},\h_t)$. Therefore, the Morse index $N(\hs)$ increases by the multiplicity of $\lambda^{2q,0}(t)$, which is the dimension of the corresponding eigenspace $E_{2q}^0$. Recall this space is non-trivial by Lemma~\ref{lemma:tanno24n+3}. In fact, it coincides with the space of functions on $S^{4n+3}$ obtained by extending the eigenfunctions corresponding to the $q^{th}$ eigenvalue of the Laplacian on the base $\Hr P^n$ to be constant along the fibers. Thus, $\dim E_{2q}^0$ also coincides with the multiplicity $m_q(\Hr P^n)$ of the $q^{th}$ eigenvalue of the quaternionic projective space $\Hr P^n$, concluding the proof.
\end{proof}

\begin{theorem}\label{thm:s4n+3}
There exists a sequence $\{t_q^{\h}\}$ given by \eqref{eq:tqh}, with $t_q^{\h}\to 0$ as $q\to\infty$, of bifurcation values for $(S^{4n+3},\h_t)$. The family is locally rigid at any $t\not\in\{t_q^{\h}\}$.
\end{theorem}

\begin{proof}
From Proposition~\ref{prop:morseindexh}, at every degeneracy value $t_q^{\h}$, $q>0$, the Morse index $N(\hs)$ increases by $m_q(\Hr P^n)>0$. Therefore, by Proposition~\ref{prop:bifmorseindex}, every degeneracy value $t_q^{\h}$, $q>0$, is a bifurcation value. In addition, $t_0^{\h}=1$ is trivially a bifurcation value, since the critical points of the Hilbert-Einstein functional in the conformal class of $\h_1$ form a manifold of dimension $4n+4$. Local rigidity of $\hs$ around $t\not\in\{t_q^{\h}\}$ follows from Proposition~\ref{prop:localrigidity}.
\end{proof}

\begin{remark}\label{rem:breaksymm}
At all bifurcation values $t_q^{\h}$, a \emph{break of symmetry} occurs, in the sense that all bifurcating solutions are \emph{non-homogeneous} metrics. This follows from the classification of homogenous metrics on $S^m$, since they are pairwise non-conformal.
\end{remark}

\section{\texorpdfstring{Homogeneous spheres $(S^{15},\kt)$}{Homogeneous spheres III}}\label{sec:s15t}

\subsection{Spectrum of the Laplacian}
Analogously to the case before, from \eqref{eq:lambdakj}, all eigenvalues of $\Delta_t$ are of the form
\begin{equation*}
\lambda^{k,j}(t)=k(k+14)+\left(\tfrac{1}{t^2}-1\right)j(j+6),
\end{equation*}
for some $k,j\in\N\cup\{0\}$. Using the same observations as in Tanno~\cite{tanno,tanno2} for the previous families, it is easy to obtain the following refinements.

\begin{lemma}\label{lemma:tanno15}
For each eigenvalue $\mu_k=k(k+14)$ of $S^{15}$, denote its corresponding eigenspace by $E_k\subset L^2(S^{15},\vol_g)$. Then $E_k$ has a orthogonal decomposition in simultaneous eigenspaces
\begin{equation*}
E_k=E_k^k+E_k^{k-2}+\dots+E_k^{k-2\left\lfloor k/2\right\rfloor},
\end{equation*}
where $\left\lfloor k/2 \right\rfloor$ denotes the largest integer less than or equal to $k/2$, and for each $\psi\in E_k^j$, $\Delta_v \psi=j(j+6) \psi$. In particular, $\lambda^{k,j}(t)$ can only be an eigenvalue of $\Delta_t$ if $0\leq j\leq k$ and $k-j$ is even, i.e., the eigenvalues of $\Delta_t$ can only be among
\begin{equation}\label{eq:lambdas15}
\lambda^{k,j}(t)=k(k+14)+\left(\tfrac{1}{t^2}-1\right)j(j+6), \quad j=k,k-2,\dots,k-2\left\lfloor k/2 \right\rfloor.
\end{equation}
\end{lemma}

\begin{lemma}\label{lemma:tanno215}
The following $E_k^j$ are non-trivial
\begin{itemize}
\item[(i)] $E_k^k$, for any $k$;
\item[(ii)] $E_k^0$ if $k$ is even.
\end{itemize}
\end{lemma}

With the same techniques, we may compute the first eigenvalue of $(S^{15},\k_t)$.

\begin{proposition}\label{prop:firsteig15}
The first non-zero eigenvalue of the Laplacian of $(S^{15},\k_t)$ is
\begin{equation*}
\lambda_1(t)=\begin{cases} 32, & 0<t\leq\sqrt{\tfrac{7}{24}} \\ 8+\tfrac{7}{t^2}, & t\geq\sqrt{\tfrac{7}{24}}. \end{cases}
\end{equation*}
\end{proposition}

\subsection{Bifurcation and local rigidity}
Analogously to the case before, we have:

\begin{proposition}\label{prop:morseindexk}
The degeneracy values for $\kt$ form a sequence $\{t_q^{\k}\}$, with $t_q^{\k}\to 0$ as $q\to\infty$, given by $t_0^{\k}=1$ and for $q>0$,
\begin{equation}\label{eq:tqk}
t_q^{\k}=\sqrt{2-\tfrac{7q}{2}-\tfrac{q^2}{2}+\tfrac12\sqrt{3+(q^2+7q-4)}}.
\end{equation}
The Morse index of $\kt$ is given by
\begin{equation}\label{eq:morsek}
N(\kt)=\begin{cases} \sum_{q=1}^r m_q(S^8), & t^{\k}_{r+1}\leq t<t^{\k}_r\\
0, & t\geq t^{\k}_1,
\end{cases}
\end{equation}
where $m_q(S^8)$ is the multiplicity of the $q^{th}$ eigenvalue of $S^8$, see \cite{bgm}. In particular, it changes whenever $t$ crosses a degeneracy value $t_q^{\k}$, is positive for $t<t_1^{\k}$ and gets arbitrarily large for small $t>0$.
\end{proposition}

\begin{proof}
Totally analogous to the proof of Proposition \ref{prop:morseindexh}.
\end{proof}

\begin{theorem}\label{thm:s15}
There exists a sequence $\{t_q^{\k}\}$ given by \eqref{eq:tqk}, with $t_q^{\k}\to 0$ as $q\to\infty$, of bifurcation values for $(S^{15},\k_t)$. The family is locally rigid at any $t\not\in\{t_q^{\k}\}$.
\end{theorem}

\begin{proof}
Totally analogous to the proof of Theorem \ref{thm:s4n+3}.
\end{proof}

\begin{remark}
Analogously to Remark~\ref{rem:breaksymm}, break of symmetry also occurs at all bifurcation values $t_q^{\k}$, i.e., bifurcating solutions are non-homogeneous metrics.
\end{remark}

\section{Remarks on uniqueness and multiplicity of solutions}
\label{sec:final}

Given a compact manifold $M$, we know that a solution to the Yamabe problem exists in every conformal class of Riemannian metrics on $M$, given by a minimizer of the Hilbert-Einstein functional \eqref{eq:a}. A natural question is then to study the nature of the space of solutions to the Yamabe problem in a given conformal class. In this direction, Brendle and Marques~\cite{BreMar09} and Khuri, Marques and Schoen \cite{KhuMarSch} obtained remarkable results on compactness and non-compactness of this space of solutions.

Another interesting question is to establish whether \emph{uniqueness} of (unit volume) solutions to the Yamabe problem holds in a given conformal class. Very little is known on this problem; for instance, uniqueness was proven in the conformal class of every Einstein metric (except for the round metric on spheres), see Obata \cite{Oba72}. Our local rigidity results for the metrics $\gt$, $\hs$ and $\kt$ on spheres imply \emph{local uniqueness} of solutions. We can prove that, in the case of these three families of metrics, the uniqueness property is \emph{stable}; more precisely:

\begin{proposition}\label{prop:stability}
The following sets are open in $\mathds R^+$:
\begin{eqnarray*}
&& \mathcal G=\left\{t\in\mathds R^+\setminus\{1\}: \begin{array}{c}
\text{\rm there is a unique }g\in\Met_1^k(S^{2n+1})\cap[\gt]\\ \text{\rm with } \scal(g)=\text{\rm const.}
\end{array}\right\}, \\
&& \mathcal H=\left\{t\in\mathds R^+\setminus\{t^\h_q\}: \begin{array}{c}
\text{\rm there is a unique }g\in\Met_1^k(S^{4n+3})\cap[\hs]\\ \text{\rm with } \scal(g)=\text{\rm const.}
\end{array}\right\}, \\
&& \mathcal K=\left\{t\in\mathds R^+\setminus\{t^\k_q\}: \begin{array}{c}
\text{\rm there is a unique }g\in\Met_1^k(S^{15})\cap[\kt]\\ \text{\rm with } \scal(g)=\text{\rm const.}
\end{array}\right\}. \\
\end{eqnarray*}
\end{proposition}

\begin{proof}
Let $t_*\in\mathcal G$ be fixed, and assume by contradiction the existence of a sequence $\{t_k\}$ in $\mathds R^+$, with $\lim_{k\to\infty}t_k=t_*$ such that the conformal class $[\g_{t_k}]$ has two distinct unit volume constant scalar curvature metrics, denoted by $g^{(1)}_k$ and $g^{(2)}_k$. We observe that for all $t\ne1$, $\gt$ is \emph{not} conformally flat, and in particular its Weyl tensor does not vanish identically. By homogeneity, it therefore does not vanish anywhere. It follows that the $C^{2,\alpha}$-\emph{a priori} estimates proved by Li and Zhang~\cite{LiZhang,LiZhang2} and Marques~\cite{coda} hold for the metrics in $[\g_{t_k}]$, which implies that the set
\begin{equation*}
\Big\{g\in[\g_t],\ t\in\{t_k,k\in\mathds N\}\cup\{t_*\}: g\in\Met_1^k(S^{2n+1}),\, \scal(g)=\text{const.}\Big\}
\end{equation*}
is compact in the $C^{2}$ topology. Thus, up to subsequences, we can assume the existence of the limits
$\lim_{k\to\infty}g^{(1)}_k=g^{(1)}_\infty$ and $\lim_{k\to\infty}g^{(2)}_k=g^{(2)}_\infty$. Clearly, $g^{(i)}_\infty$ is a unit volume
constant scalar curvature in the conformal class $[\g_{t_*}]$, $i=1,2$.
By the uniqueness at $t_*$, it must be $g^{(1)}_\infty=g^{(2)}_\infty=\overline{\g}_{t_*}$, where $\overline{\g}_{t_*}$ is the unit volume
metric homothetic to $\g_{t_*}$.\footnote{Recall that the Hilbert--Einstein variational problem is invariant by renormalization of the metrics, see Remark~\ref{rem:renorm}, and so are all the results of this paper.}
On the other hand, since $\overline{\g}_{t_*}$ is a nondegenerate critical point of the Hilbert-Einstein functional,
then for $k$ large enough also $g^{(1)}_k$ and $g^{(2)}_k$ are nondegenerate, and arbitrarily close to each other.
This contradicts the local rigidity at $t_k$, see Proposition~\ref{prop:localrigidity}, and proves that $\mathcal G$ is open.
The openness of $\mathcal H$ and $\mathcal K$ follows by a totally analogous argument.
\end{proof}

Conversely, one is also interested in establishing which conformal classes carry multiple unit volume metrics with constant scalar curvature. Our bifurcation results yield the following.

\begin{proposition}\label{prop:mult}
There exists infinite subsets $\widetilde{\mathcal H}$ and $\widetilde{\mathcal K}$ of $\,\left]0,1\right[$, with $0$ in their closure, such that for all $t\in\widetilde{\mathcal H}$ (respectively $t\in\widetilde{\mathcal K}$), the conformal class $[\hs]$ (respectively $[\kt]$) has at least $3$ distinct unit volume constant scalar curvature metrics.
\end{proposition}

\begin{proof}
For $k>1$, consider the bifurcation value $t^\h_k$ for the family $\hs$, see Theorem~\ref{thm:s4n+3}.
For all $t$, denote by $\overline\h_t$ the unit volume metric homothetic to $\hs$. Arbitrarily close to $t^\h_k$ there are values $t$ such that the conformal class $[\hs]$ contains a unit volume constant scalar curvature metric $\widetilde h$ distinct from $\overline\h_t$. Since the Morse index of $\overline\h_t$ is positive (Proposition~\ref{prop:morseindexh}), by continuity, also the Morse index of $\widetilde h$ is positive. In particular, neither $\overline\h_t$ nor $\widetilde h$ are minima of the Hilbert-Einstein functional. Therefore, $[\hs]$ contains at least $3$ distinct unit volume constant scalar curvature metrics. Clearly, such $t$'s accumulate at $0$, since $\lim_{k\to\infty}t^\h_k=0$, see Theorem \ref{thm:s4n+3}. The argument for $\kt$ is totally analogous, using Theorem~\ref{thm:s15} and Proposition~\ref{prop:morseindexk}.
\end{proof}

\end{document}